\documentclass[11pt,oneside]{amsart}

\usepackage{amsmath,ifthen, amsfonts, amssymb,
srcltx, amsopn, color, enumerate, mathrsfs, overpic}
\usepackage[cmtip,arrow]{xy}
\usepackage{pb-diagram, pb-xy}
\usepackage{overpic}
\usepackage{subfig}

\usepackage[T1]{fontenc}
\usepackage{yfonts}

\usepackage{graphicx}

\newcommand{\showcomments}{yes}
\renewcommand{\showcomments}{no}

\newsavebox{\commentbox}
\newenvironment{com}%
{\ifthenelse{\equal{\showcomments}{yes}}%
{\footnotemark
        \begin{lrbox}{\commentbox}
        \begin{minipage}[t]{1.25in}\raggedright\sffamily\tiny
        \footnotemark[\arabic{footnote}]}
{\begin{lrbox}{\commentbox}}}%
{\ifthenelse{\equal{\showcomments}{yes}}%
{\end{minipage}\end{lrbox}\marginpar{\usebox{\commentbox}}}
{\end{lrbox}}}

\newtheorem{thm}{Theorem}[section]

\newtheorem{lem}[thm]{Lemma}

\newtheorem{cor}[thm]{Corollary}

\newtheorem{prop}[thm]{Proposition}

\theoremstyle{definition}
\newtheorem{defn}[thm]{Definition}
\newtheorem{rem}[thm]{Remark}

\newtheorem{claim*}{Claim}

\DeclareMathOperator{\dimension}{dim}

\DeclareMathOperator{\link}{lk}

\DeclareMathOperator{\stabilizer}{Stab}

\DeclareMathOperator{\SL}{SL}

\newcommand{\Z}{\ensuremath{\mathbb{Z}}}

\newcommand{\field}[1]{\mathbb{#1}}
\newcommand{\integers}{\ensuremath{\field{Z}}}

\newcommand{\Euclidean}{\ensuremath{\field{E}}}

\newcommand{\interior} [1] {{\ensuremath \text{\rm Int}(#1) }}

\makeatletter

\newcommand{\Rmnum}[1]{\mathbf{{\expandafter\@slowromancap\romannumeral #1@}}}

\newcommand{\lcm}{\ensuremath{\mathrm{l}\mathrm{c}\mathrm{m}}}

\makeatother

\let\oldmarginpar\marginpar
\renewcommand\marginpar[1]{\-\oldmarginpar[\raggedleft\footnotesize #1]%
{\raggedright\footnotesize #1}}

\newcounter{enumitemp}

\newcommand{\N}{\ensuremath{\mathbb{N}}}

\newcommand{\dist}{\textup{\textsf{d}}}
\DeclareMathOperator{\F}{RF}
\DeclareMathOperator{\DM}{F}
\DeclareMathOperator{\verts}{\text{Vertices}}
\DeclareMathOperator{\edges}{\text{Edges}}

\newcommand{\cancom}[2]{\complement\left({#1}\rightarrow{#2}\right)}

\begin{document}
\title{Residual finiteness growths of virtually special groups}
\author[K. Bou-Rabee]{Khalid bou-Rabee}
\address{School of Mathematics, University of Minnesota--Twin Cities, Minneapolis, Minnesota, USA}
\email{khalid.math@gmail.com}
\author[M.F. Hagen]{Mark F. Hagen}
\address{Department of Mathematics, University of Michigan, Ann Arbor, Michigan, USA}
\email{markfhagen@gmail.com}
\author[P. Patel]{Priyam Patel}
\address{Department of Mathematics, Purdue University, West Lafayette, Indiana, USA}
\email{patel376@math.purdue.edu}
\date{\today}

\begin{abstract}
Let $G$ be a virtually special group. Then the residual finiteness growth of $G$ is at most linear.
This result cannot be found by embedding $G$ into a special linear group.
Indeed, the special linear group $\SL_k(\Z)$, for $k > 2$, has residual finiteness growth $n^{k-1}$.
\end{abstract}
\subjclass[2010]{Primary: 20E26; Secondary: 20F65, 20F36}
\keywords{Residual finiteness growth, special cube complex, right-angled Artin group}
\maketitle
\setcounter{tocdepth}{2}

\tableofcontents

\section{Introduction}\label{sec:introduction}

This paper demonstrates that Stalling's topological proof of the residual finiteness  of free groups \cite[Theorem 6.1]{stallings} extends efficiently to the class of right-angled Artin groups (and, more generally, to virtually special groups).
To state our results, we use notation developed in \cite{BM10}, \cite{BM13}, \cite{MR2851069} for quantifying residual finiteness.
Let $A$ be a residually finite group with generating set $X$.
The \emph{divisibility function} $D_A : A \setminus\{1\} \to \N$ is defined by
$$
D_A(g) = \min \{ [ A : B] : g \notin B \wedge B \leq A \}. $$

Define $\DM_{A,X}(n)$ to be the maximal value of $D_A$ on the set
$$
\left\{ g :\| g  \|_X \leq n \: \wedge \: g \neq 1 \right \},
$$
\noindent
where $\| \cdot \|_X$ is the word-length norm with respect to $X$.  The growth of $\DM_{A,X}$ is called the \emph{residual finiteness growth function}.  The growth of $\DM_{A,X}(n)$ is, up to a natural equivalence, independent of the choice of generating set (see \S \ref{sec:quantifying} for this and generalizations of $\DM_{A,X}$).
Loosely speaking, residual finiteness growth measures how well a group is approximated by its profinite completion.

Our first result, proved in \S \ref{sec:separability}, gives bounds on the residual finiteness growth of any right-angled Artin group:

\begin{thm}\label{thm:ssraag}\label{thm:introraag}
Let $\Gamma$ be a simplicial graph.
Let $A_{\Gamma}$ be the corresponding right-angled Artin group with standard generating set $X$.
Then $\DM_{A_\Gamma, X}(n) \leq n+1$.
\end{thm}
\noindent
The canonical completion \cite{haglundwise}, a pillar from the structure theory of special groups, plays an integral role in our proof of Theorem \ref{thm:ssraag}.
Thus, we include a complete proof of a quantified version of the canonical completion construction in \S \ref{subsubsec:can_com}.

Our residual finiteness bound for right-angled Artin groups, hereafter known as \emph{raAgs}, is an \emph{efficient} extension of Stalling's proof of \cite[Theorem 6.1]{stallings}.
Indeed, the bound achieved over the class of all raAgs exactly coincides with that given by a direct application of Stallings' proof in the case $A_\Gamma = \Z * \Z$.
Further, Theorem \ref{thm:ssraag} immediately extends to many other groups.
That is, since bounds on residual finiteness growth are inherited by passing to subgroups and to super groups of finite-index (see \S\ref{sec:quantifying}),
Theorem~\ref{thm:introraag} gives bounds on residual finiteness growth for any group that virtually embeds into a raAg, the so-called \emph{virtually special groups} (defined in Section \ref{subsubsec:salvetti}).
This class includes, for example, Coxeter groups \cite{MR1983376, MR2646113}, free-by-$\Z$ groups with irreducible atoroidal monodromy \cite{hagen_wise_irred}, hyperbolic 3-manifold groups \cite{Wise:QCH, Agol:virtual_haken},  fundamental groups of many aspherical graph manifolds \cite{PrzytyckiWise:cube_GraphManifold, liu_graph}, fundamental groups of mixed 3-manifolds \cite{PrzytyckiWise:mixed}, and random groups at sufficiently low density \cite{Agol:virtual_haken, OllivierWise:Density}.
See also \cite{Wilton} for more explicit results in the 3-manifold case.

Our second result gives a sense of the topological nature of Theorem~\ref{thm:introraag}. It is known that any finitely generated raAg embeds in $\SL_k(\Z)$ for appropriately-chosen $k$, and that $k>2$ unless the raAg in question is free \cite{MR1815219}, \cite{MR1888796}, \cite{MR1942303}.  However, Theorem~\ref{thm:arithmetic} shows that the best upper bound on residual finiteness growth of the raAg that can be inferred in this manner is superlinear (c.f. \cite{BK12}, where the normal residual finiteness growth for $\SL_k(\Z), k>2,$ is shown to be $n^{k^2-1}$).

\begin{thm} \label{thm:arithmetic}
The residual finiteness growth of $\SL_k(\Z)$, $k > 2$, is bounded above by $Cn^{k-1}$ and below by $1/C n^{k-1}$ for some fixed $C > 0$.
That is, $\DM_{\SL_k(\Z)} (n) \simeq n^{k-1}$.
\end{thm}
\noindent
Since residual finiteness growth is a commensurability invariant (see \S \ref{sec:quantifying}), this theorem computes the residual finiteness growth for any $S$ arithmetic subgroup of $\SL_k(\mathbb{C})$ for $k > 2$ (in the sense of \cite{BK12}).
Our proof of Theorem \ref{thm:arithmetic} relies on a result of Lubotzky, Mozes, and Raghunathan \cite{LMR00} and the structure theory of finite-index subgroups of unipotent groups \cite{MR943928}.

This article is organized as follows. In \S \ref{sec:background}, we give some necessary background on quantifying residual finiteness, raAgs, and cubical geometry. In the interest of self-containment, we also provide the construction of the canonical completion from \cite{haglundwise} for a special case that is relevant to our framework. In \S \ref{sec:separability}, we generalize Stallings' proof to raAgs to give a proof of Theorem \ref{thm:introraag}.
In \S \ref{sec:arithmetic}, we conclude with a proof of Theorem \ref{thm:arithmetic}.

\subsection*{Acknowledgments}
The authors are grateful to Ian Agol, Benson Farb, Michael Larsen, Feng Luo, and Ben McReynolds for useful and stimulating conversations. K.B. and M.F.H. gratefully acknowledge the hospitality and support given to them from the Ventotene 2013 conference for a week while they worked on some of the material in this paper. K.B. was partially supported by the NSF grant DMS-1405609. The work of M.F.H. was supported by the National Science Foundation under Grant Number NSF 1045119. Finally, the authors would like to thank the referee for several excellent suggestions and comments that have improved the quality of this paper.

\section{Background}\label{sec:background}
\subsection{Quantifying residual finiteness}\label{sec:quantifying}
For previous works on quantifying this basic property see, for instance, \cite{MR2583614}, \cite{PP12}, \cite{MR2784792}, \cite{BM13}, \cite{BK12}, \cite{BM11}, \cite{B11}, \cite{BM10}, and \cite{R12}.
Here we unify the notation used in previous papers and demonstrate that this unification preserves basic properties of residual finiteness growth.
Further, we will see that this unification also elucidates the behavior of residual finiteness functions under commensurability.

Let $A$ be a group with generating set $X$.
Given a class of subgroups $\Omega$ of $A$,
set $\Lambda_\Omega = \cap_{\Delta \in \Omega} \Delta$ (c.f. \cite{BBKM13b}).
The \emph{divisibility function} (c.f. \cite{BM10}, \cite{BM13}, \cite{MR2851069}), $D_A^{\Omega} : A \setminus \Lambda_\Omega \to \N$ is defined by
$$
D_A^{\Omega}(g) = \min \{ [ A : B] : g \notin B \wedge B \in \Omega\}. $$

Define $\F_{A,X}^{\Omega}(n)$ to be the maximal value of $D_A^\Omega$ on the set
$$
\left\{ g :\| g  \|_X \leq n \: \wedge \: g \notin \Lambda_\Omega  \right \},
$$
where $\| \cdot \|_X$ is the word-length norm with respect to $X$.  The growth of $\F_{A,X}^\Omega$ is called the \emph{residual $\Omega$ growth function}.  The growth of $\F_{A,X}^{\Omega}(n)$ is, up to a natural equivalence, independent of the choice of generating set (see Lemmas \ref{lem:inherits} and \ref{lem:supergrouplinear} below).

When $\Omega$ is the class of all subgroups of a residually finite group, $A$, we have $\Lambda_\Omega = \{ 1 \}$ and the growth of $\F_{A,X}^{\Omega}(n)$ is  the \emph{residual finiteness growth} of $A$.
In the case when $\Omega$ consists of all normal subgroups of $A$ and $\Lambda_\Omega = \{1 \}$, the function $\F_{A,X}^\Omega$ is the \emph{normal residual finiteness growth function}.

Our first result demonstrates that the residual $\Omega$ growth of a group is well-behaved under passing to subgroups.
\begin{lem}\label{lem:inherits}
Let $G$ be generated by a set $S$, and let $H\leq G$ be generated by a finite set $L\subset G$.
Let $\Omega$ be a class of subgroups of $G$.
Then there exists $C > 0$ such that $\F_{H,L}^{\Omega \cap H}(n)\leq\F_{G,S}^\Omega(Cn)$ for all $n\geq 1$.
\end{lem}

\begin{proof}
By definition of $\Omega$ and $\Omega \cap H$ we have $D_H^{\Omega \cap H}(h) \leq D^{\Omega}_G(h)$ for all $h \in H$ and $h \notin \Lambda_{\Omega \cap H} = \Lambda_\Omega \cap H$.
Hence,

\begin{align} 
\F_{H,L}^{\Omega \cap H}(n) &= \max\{  D_H^{\Omega \cap H} (h) \: :\: \| h \|_L \leq n \wedge h \notin \Lambda_{\Omega \cap H}\} \label{firstlemmaeq1} \\ 
&\leq \max\{ D_G^{\Omega} (h) \:: \: \| h \|_L \leq n \wedge h \notin \Lambda_\Omega \}. \nonumber
\end{align}

\noindent
Further, there exists a $C>0$ such that any element in $L$ can be written in terms of at most $C$ elements of $S$.
Thus,

\begin{equation} \label{firstlemmaeq2}
\{ h \in H : \| h \|_L \leq n \} \subseteq \{ g \in G : \| g \|_S \leq Cn \}.
\end{equation}

\noindent
So by (\ref{firstlemmaeq1}) and (\ref{firstlemmaeq2}), we have that
$$
\F_{H,L}^{\Omega \cap H}(n) \leq
\max\{  D_G^{\Omega} (h) : \| h \|_L \leq n \}
\leq
\max\{  D_G^{\Omega} (g) :  \| g \|_S \leq C n \}
= \F_{G,S}^{\Omega}(C n),$$
as desired.
\end{proof}

Lemma \ref{lem:inherits} shows that the residual $\Omega$ growth of a group does not depend heavily on the choice of generating set.
To this end, we write $f \preceq g$ if there exists $C > 0$ such that $f(n) \leq C g (Cn)$. Further, we write $f \simeq g$ if $f \preceq g$ and $g \preceq f$.
If $f \preceq g$, we say that $g$ \emph{dominates} $f$.
So, in particular, Lemma \ref{lem:inherits} gives that, up to $\simeq$ equivalence, the growth of $\F_{A,X}^\Omega$ does not depend on the choice of generating set.
We can and will often, therefore, drop the decoration indicating the generating set $X$ when dealing with growth functions.

The next result, coupled with Lemma \ref{lem:inherits}, demonstrates that residual $\Omega$ growth is also well-behaved under passing to super groups of finite-index.

\begin{lem} \label{lem:supergrouplinear}
Let $H$ be a finite-index subgroup of a finitely generated group $G$.
Let $H$ be generated by $L$ and $G$ by $S$.
Let $\Omega$ be a class of subgroups of $G$ with $H \in \Omega$.
Then $\F_G^\Omega(n) \preceq \F_H^{\Omega \cap H}(n)$.
\end{lem}

\begin{proof}
$D_G^\Omega(g) \leq [G:H] D_H^{\Omega \cap H}(g)$ if $g \in H$ and because $H \in \Omega$, $D_G^\Omega (g) \leq [G:H]$ for $g \notin H$.
Bringing these facts together gives
\begin{equation} \label{secondlemmaeq1}
\F_{G,S}^{\Omega}(n) = \max\{  D_G^{\Omega} (g) \: :\: \| g \|_S \leq n \} \leq
[G:H]  \max\{  D_H^{\Omega \cap H} (g) \:: \: \| g \|_S \leq n \wedge g \in H \setminus \Lambda_\Omega \}.
\end{equation}

Fix word metrics $d_H$ and $d_G$ for $H$ and $G$ with respect to $L$ and $S$ respectively.
The identity map $H \to G$ is a $(K,0)$ quasi-isometry by the Milnor-Schwarz lemma.
For any element $g \in G$ with $d_G(g,1) \leq n$, either $g \notin H$ or $d_H(g,1) \leq K d_G(g,1) \leq K n$.
Thus, there exists a natural number $C$ such that for all $n$,
\begin{equation} \label{secondlemmaeq2}
\{ h \in H : \| h \|_S \leq n \} \subseteq \{ g \in H : \| g \|_L \leq Cn \}.
\end{equation}
So by (\ref{secondlemmaeq1}) and (\ref{secondlemmaeq2}), we have that
\begin{eqnarray*}
\F_{G,S}^{\Omega}(n) &\leq&
[G:H] \max\{  D_H^{\Omega \cap H} (g) : \| g \|_S \leq n \wedge g \in H \setminus \Lambda_\Omega \} \\
&\leq&
[G:H] \max\{  D_H^{\Omega \cap H} (g) :  \| g \|_L \leq Cn  \wedge g \in H \setminus \Lambda_\Omega \}.
\end{eqnarray*}
Thus,
$\F_{G,S}^{\Omega}(n) \leq [G:H] \F_{H,L}^{\Omega \cap H} (C n),$
as desired.
\end{proof}

Recall that subgroups $G$ and $H$ of $A$ are \emph{commensurable} if $G \cap H$ is finite-index in both $G$ and $H$.
  Lemma \ref{lem:supergrouplinear} and Lemma \ref{lem:inherits} demonstrate that residual $\Omega$ growth, and residual finiteness growth, behave well under commensurability as noted in the following Proposition.

\begin{prop} \label{prop:comm}
Let $G$ be a finitely generated subgroup of $A$.
Let $\Omega$ be a class of subgroups of $G$.
Let $H$ be commensurable with $G$, and let $G \cap H \in \Omega$.
Then $\F_{H}^{\Omega \cap H} (n) \simeq \F_{G}^{\Omega}(n)$.
In particular, $\DM_G(n) \simeq \DM_{H}(n)$.
\end{prop}

\begin{proof}
Since $G \cap H$ is a finite-index subgroup of both $G$ and $H$, Lemma \ref{lem:supergrouplinear} gives that
$\F_{G \cap H}^{\Omega \cap H}(n)$  dominates both $\F_H^{\Omega \cap H}(n)$ and $\F_G^{\Omega}(n)$.
Further, Lemma \ref{lem:inherits}, demonstrates that both $\F_H^{\Omega \cap H}(n)$ and $\F_G^{\Omega}(n)$  dominate
$\F_{G \cap H}^{\Omega \cap H}(n).$
Thus,
$\F_H^{\Omega \cap H} (n) \simeq \F_{H}^{\Omega \cap H}(n) \simeq \F_G^\Omega (n),$
as desired.

In the case when $\Omega$ is the set of all finite-index subgroups of $G$, we see that $\Omega \cap G \cap H$ is precisely the set of all finite-index subgroups of $G \cap H$.
Thus, we have
$\DM_G(n) \preceq \DM_{G \cap H} (n)$.
So as $G \cap H$ is finite-index in $H$, we similarly achieve
$\DM_H(n) \preceq \DM_{G \cap H} (n)$.
So by Lemma \ref{lem:inherits}, we are done.
\end{proof}

\subsection{Right-angled Artin groups and virtually special groups}\label{subsec:raag_background}

\subsubsection{Right-angled Artin groups}\label{subsubsec:raags}

Right-angled Artin groups (raAgs) are a widely-studied class of groups  (see~\cite{charney_raag_survey} for a comprehensive survey).
These groups  have great utility in geometric group theory both because the class of subgroups of raAgs is very rich and because raAgs are fundamental groups of particularly nice nonpositively-curved cube complexes.
For each finite simplicial graph $\Gamma$, the associated finitely generated raAg $A_{\Gamma}$ is given by the presentation $$\left\langle x_i \in \verts(\Gamma)\,\mid\, [x_j, x_k] = x_jx_kx_j^{-1}x_k^{-1},\,\{x_j,x_k\}\in\edges(\Gamma)\right\rangle.$$

For example, if $\Gamma$ has no edges, then $A_{\Gamma}$ is the free group, freely generated by $\verts(\Gamma)$, while $A_{\Gamma}\cong\integers^{|\verts(\Gamma)|}$ when $\Gamma$ is a clique.  More generally, $A_{\Gamma}$ decomposes as the free product of the raAgs associated to the various components of $\Gamma$, and if $\Gamma$ is the join of subgraphs $\Gamma_1,\Gamma_2$, then $A_{\Gamma}\cong A_{\Gamma_1}\times A_{\Gamma_2}$.

\subsubsection{Nonpositively-curved cube complexes}\label{subsubsec:npccc}
We recall some basic notions about nonpositively-curved cube complexes that will be required below.  More comprehensive discussions of CAT(0) and nonpositively-curved cube complexes can be found in, e.g.,~\cite{Chepoi:cube_median,HagenPhD,Sageev:cubes_95,Wise:QCH,cbmsnotes}.  We largely follow the discussion in~\cite{cbmsnotes}.

For $d\geq 0$, a \emph{$d$-cube} is a metric space isometric to $[-\frac{1}{2},\frac{1}{2}]^d$ with its $\ell^1$ metric.  A \emph{$d'$-face} of the $d$-cube $C$ is a subspace obtained by restricting $d-d'$ of the coordinates to $\pm\frac{1}{2}$, while a \emph{midcube} of $C$ is a subspace obtained by restricting exactly one coordinate to 0.  A \emph{cube complex} is a CW-complex whose cells are cubes of various dimensions and whose attaching maps restrict to combinatorial isometries on faces.

Let $X$ be a cube complex and let $x\in X$ be a 0-cube.  The \emph{link} $\link(x)$ of $x$ is the simplicial complex with an $n$-simplex $\sigma_c$ for each $(n+1)$-cube $c$ containing $x$, with the property that $\sigma_{c}\cap\sigma_{c'}=\cup_{c''}\sigma_{c''}$, where $c''$ varies over the constituent cubes of $c\cap c'$.  A simplicial complex is \emph{flag} if each $(n+1)$-clique in the 1-skeleton spans an $n$-simplex, and $X$ is \emph{nonpositively-curved} if $\link(x)$ is flag for each $x\in X^{(0)}$.  If the nonpositively-curved cube complex $X$ is simply connected, then $X$ is a \emph{CAT(0)} cube complex, so named for the existence of a natural piecewise-Euclidean CAT(0) metric~\cite{gromov, bridson, leary}.

All maps of nonpositively-curved cube complexes in this paper are, unless stated otherwise, \emph{cubical maps}, i.e. they send open $d$-cubes isomorphically to open $d$-cubes for $d>0$ and send 0-cubes to 0-cubes.

\subsubsection{Special cube complexes}\label{subsubsec:salvetti}
Special cube complexes were defined in~\cite{haglundwise} in terms of the absence of certain pathological configurations of immersed hyperplanes.  Here, we are interested in the characterization of special cube complexes in terms of raAgs, established in the same paper.

Let $\Gamma$ be a simplicial graph and $A_{\Gamma}$ the associated raAg.  The \emph{Salvetti complex} $S_{\Gamma}$ associated to $\Gamma$ is a $K(A_{\Gamma},1)$ cube complex, first constructed in~\cite{charney_davis}, that we now describe.  $S_{\Gamma}$ has a single 0-cube $v$ and a 1-cube $e_{x_i}$ for each $x_i\in\verts(\Gamma)$.  For each relation $[x_j,x_k]$ in the above presentation of $A_{\Gamma}$, we add a 2-cube with attaching map $e_{x_j}e_{x_k}e_{x_j}^{-1}e_{x_k}^{-1}$.  Finally, we add an $n$-cube for each size-$n$ set of pairwise-commuting generators.  Note that the image in $S_{\Gamma}$ of each $n$-cube is an embedded $n$-torus and $S_{\Gamma}$ is a nonpositively-curved cube complex.

The cubical map $f:Y\rightarrow X$ of nonpositively-curved cube complexes is a \emph{local isometry} if the following conditions are satisfied:
\begin{enumerate}
\item $f$ is locally injective; equivalently, the induced map $\link(x)\rightarrow\link(f(x))$ is injective for each $x\in X^{(0)}$, and
\item for each $x\in X^{(0)}$, the subcomplex $\link(x)\subseteq\link(f(x))$ is an \emph{induced} subcomplex in the sense that $n+1$ vertices of $\link(x)$ span an $n$-simplex whenever their images in $\link(f(x))$ span an $n$-simplex.
\end{enumerate}
If $X, Y$ are CAT(0) cube complexes, and there is an injective local isometry $Y\rightarrow X$, then $Y$ is \emph{convex} in $X$.  More generally, if $X,Y$ are nonpositively-curved and there is a local isometry $Y\rightarrow X$, then the image of $Y$ is \emph{locally convex} in $X$. It should be noted that covering maps of nonpositively-curved cube complexes are local isometries.

\begin{rem}[Cubical convexity]\label{rem:geometry_in_x}
The term ``convex'' is justified by the fact that if $Y$ is a convex subcomplex of the CAT(0) cube complex $X$, in the preceding sense, then $Y^{(1)}$, with the usual graph metric, is metrically convex in $X^{(1)}$.  When working with a CAT(0) cube complex $X$, we will only use the usual graph metric on $X^{(1)}$, ignoring the CAT(0) metric.  In particular, a \emph{(combinatorial) geodesic} in $X$ shall be understood to be a path in $X^{(1)}$ with a minimal number of edges among all paths with the given endpoints.  Equivalently, a combinatorial path $\gamma\rightarrow X$ is a geodesic if and only if each hyperplane of $X$ intersects at most one 1-cube of $\gamma$, and a connected subcomplex $Y$ of $X$ has isometrically embedded 1-skeleton if and only if it has connected intersection with each hyperplane.  We will refer to a connected subcomplex of $X$ as \emph{isometrically embedded} if it has the latter property.
\end{rem}

The notion of a cubical local isometry yields an elegant characterization of special cube complexes (see~\cite{haglundwise}) which we shall take to be the definition:

\begin{defn}[Special cube complex, virtually special group]\label{defn:special}
The nonpositively-curved cube complex $X$ is \emph{special} if there exists a local isometry $X\rightarrow S_{\Gamma}$ for some simplicial graph $\Gamma$.  The group $G$ is \emph{[virtually] special} if there exists a special cube complex $X$ having [a finite-index subgroup of] $G$ as its fundamental group.  If this cube complex can be chosen to be compact, then $G$ is \emph{virtually compact special}.
\end{defn}

\subsubsection{Canonical completion}\label{subsubsec:can_com}
A substantial part of the utility of special cube complexes is the fact that they behave in several important ways like graphs.  Chief among the graph-like features of special cube complexes is the ability to extend compact local isometries to covers, generalizing the fact that finite immersions of graphs extend to covering maps~\cite{stallings}.  This procedure, introduced in~\cite{haglundwise} and outlined presently, is called ``canonical completion''.  Since it is more directly suited to our situation, we will follow the discussion in~\cite{haglund_wise_amalgams}; in the interest of a relatively self-contained exposition, we now sketch the special case of the construction in~\cite[Definition~3.2]{haglund_wise_amalgams} that we will later require.

\begin{thm}\label{thm:cancom}[Canonical completion for Salvetti complexes]
Let $Y$ be a compact cube complex, and let $f:Y\rightarrow S_\Gamma$ be a local isometry, where $S_\Gamma$ is the Salvetti complex of a raAg $A_\Gamma$.  Then there exists a finite-sheeted cover $\widehat S_\Gamma \rightarrow S_\Gamma$ such that $f$ lifts to an embedding $\hat f:Y\rightarrow\widehat S_\Gamma$.
\end{thm}

The space $\widehat S_\Gamma$ is called the \emph{canonical completion} of $f$ and will be denoted by $\cancom{Y}{S_\Gamma}$.

\begin{proof}[Proof of Thm.~\ref{thm:cancom}]
Let $e$ be a (closed) oriented 1-cube of $S_{\Gamma}$.  Each component of the preimage of $e$ in $Y$ is either a cycle, an interval, or a $0$-cube mapping to the base-point, since $Y\rightarrow S_{\Gamma}$ is locally injective.  For each non-cycle component, we add an appropriately oriented open 1-cube to $Y$ to form an oriented cycle covering $e$. The map $f: Y\rightarrow S_{\Gamma}$ extends by declaring the new open 1-cube to map by an orientation-preserving homeomorphism to $\interior{e}$.  Let $Y^\circ$ be the union of $Y$ and all of these new 1-cubes. We thus have a map $\hat f: Y^{\circ}\rightarrow S_{\Gamma}$ that extends $f$ and is a covering map on 1-skeleta. The 1-skeleton of $Y^\circ$ will be the 1-skeleton of $\cancom{Y}{S_\Gamma}$. By Lemma~\ref{lem:Khalids_picture} below, for each 2-cube $c$ of $S_{\Gamma}$, the boundary path of $c$ lifts to a closed path in $Y^{\circ}$, and hence we can attach 2-cubes to $Y$ to form a complex $Y^{\bullet}$ equipped with a cubical map $Y^{\bullet}\rightarrow S_{\Gamma}$ that extends $Y^{\circ}\rightarrow S_{\Gamma}$ and restricts to a covering map on 2-skeleta.  For each higher-dimensional cube $c$ of $S_{\Gamma}$, the 2-skeleton of $c$ lifts to $Y^{\bullet}$, and we form $\cancom{Y}{S_{\Gamma}}$ by adding to $Y^{\bullet}$ each cube whose 2-skeleton appears.  By construction, $\cancom{Y}{S_{\Gamma}}$ covers $S_{\Gamma}$ and is thus non-positively curved.
\end{proof}

\begin{lem}\label{lem:Khalids_picture}
For each 2-cube $c$ of $S_{\Gamma}$, the boundary path
of $c$ lifts to a closed path in $Y^{\circ}$.
\end{lem}

\begin{proof}
Let $f : Y \to S_\Gamma$ be the local isometry from the proof of Theorem \ref{thm:cancom}, and also denote by $f$ its extension to $Y^\circ$.
Let $\overline{\gamma} : [0, 4] \to S_\Gamma$ be the 
boundary path of $c$ and let $\gamma$ be a lift of $\overline{\gamma}$ through $f$.

There exists, by construction, a finite sequence 
$( C_i )_{i=1}^4$ of (not necessarily distinct) cycles in $Y^{\circ}$ such that
\begin{enumerate}
\item $\forall i$, $f(C_i )$ is a 1-cube in $S_\Gamma$;
\item $(C_1\cap \gamma) \cdot (C_2 \cap \gamma) \cdot (C_3 \cap \gamma) \cdot (C_4 \cap \gamma) = \gamma$;  and
\item each $C_i$ has at most one 1-cube not in $Y$.
\end{enumerate}
Set  $\gamma_i = \gamma([i-1, i]) = C_i \cap \gamma$ and set $C_i^c := C_i \setminus \text{int}(\gamma_i)$. We note that when $\gamma_i$ is a loop, $C_i^c$ 
consists of the single vertex $\gamma(i)$.

\begin{figure}[h]
\begin{overpic}[width=.5\textwidth]{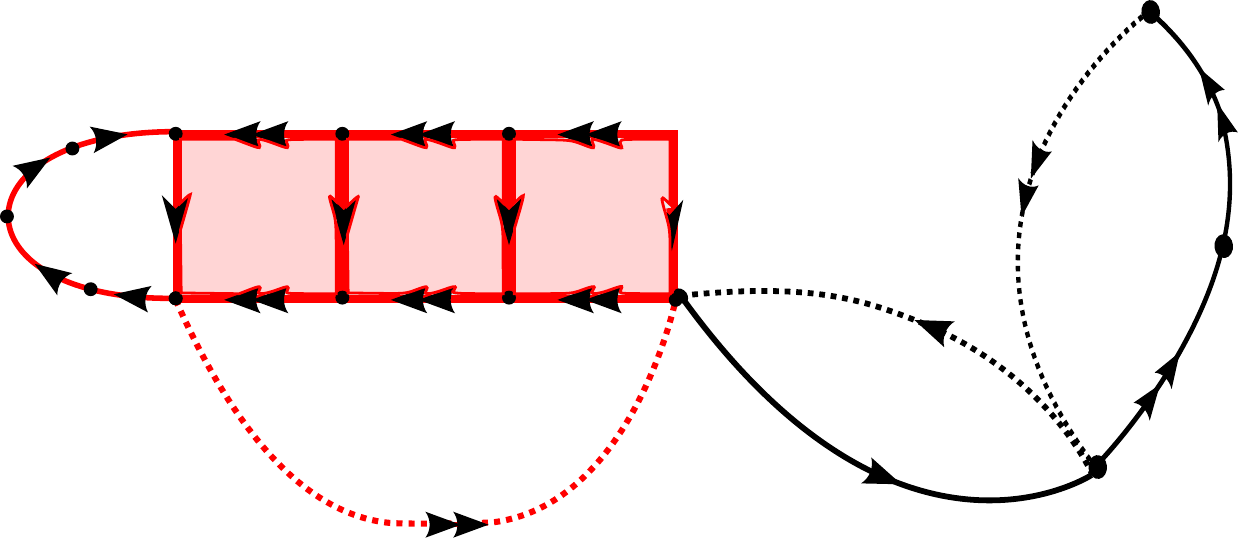}
\put(10,22){\scriptsize{$\mathbf {D_1=\gamma_1}$}}
\put(2,24){\scriptsize{$\mathbf{C_1}$}}
\put(30,15){\scriptsize{$\mathbf {D_2}$}}
\put(30,5){\scriptsize{$\mathbf {\gamma_2}$}}
\put(70,20){\scriptsize{$\mathbf {\gamma_3}$}}
\put(78, 25){\scriptsize{$\mathbf {\gamma_4}$}}
\put(92, 25){\scriptsize{$\mathbf {C_4}$}}
\put(70,8){\scriptsize{$\mathbf {C_3}$}}
\end{overpic}
\caption{}\label{fig:figure_p1}
\end{figure}

{\bf Case 1a:} Suppose that $\gamma_1$ is not a loop in $Y^\circ$ so that $C_1^c$ is not a vertex. Suppose that $\gamma_2$ is also not a loop. Condition (3) above then ensures that either $\gamma_1 \subset Y$, in which case we set $D_1 = \gamma_1$, or $\gamma_1 \subset Y^\circ \setminus Y$ so that $C_1^c \subset Y$ and we set $D_1 = C_1^c$. Similarly, either $\gamma_2 \subset Y$ and $\gamma_2 = D_2$, or $D_2 = C_2^c$. Since $f: Y \rightarrow S_\Gamma$ is a local isometry and $D_1 \cup D_2 \subset Y$, the map $D_1 \cup D_{2} \hookrightarrow Y$ extends to a map $D_1 \times D_{2} \hookrightarrow Y$. Let $\ell_i$ denote the length of the path $D_i$ for $i = 1, 2$.  The map $D_1\times D_2\rightarrow Y$ is a cubical embedding of the standard tiling of $[0,\ell_1]\times[0,\ell_2]$ by 2-cubes. Since $f|: (Y^\circ)^{(1)} \rightarrow S_\Gamma^{(1)}$ is an immersion, the third side of the rectangle $D_1 \times D_2$ must coincide with either $\gamma_3$ or $C_3^c$ (depending on the orientation of the third side of $D_1 \times D_2$) as shown by Figure~\ref{fig:figure_p1}. If $D_1 \times D_2 \cap \gamma_3 \neq \emptyset$ we set $D_3 = \gamma_3$. Otherwise, we set $D_3 = C_3^c$.

The fact that $f$ is a covering map on the 1-skeleta of $Y^\circ$ and $S_\Gamma$ implies that $\ell_1 = \ell_3$, where $\ell_3$ is the length of $D_3$. Indeed, Figure~\ref{fig:figure_p2} shows that if $\ell_1 > \ell_3$, the covering map condition fails at the vertex $v_0 \in f^{-1}(v)$. Similarly, Figure~\ref{fig:figure_p3} shows that if $\ell_1 < \ell_3$, then the covering criterion would fail at the vertex $v_1 \in f^{-1}(v)$. Thus $D_3$ is precisely the third side of $D_1 \times D_2$.

\begin{figure}[h]
\begin{minipage}[b]{.4\textwidth}
\centering
\begin{overpic}[width=0.7\textwidth]{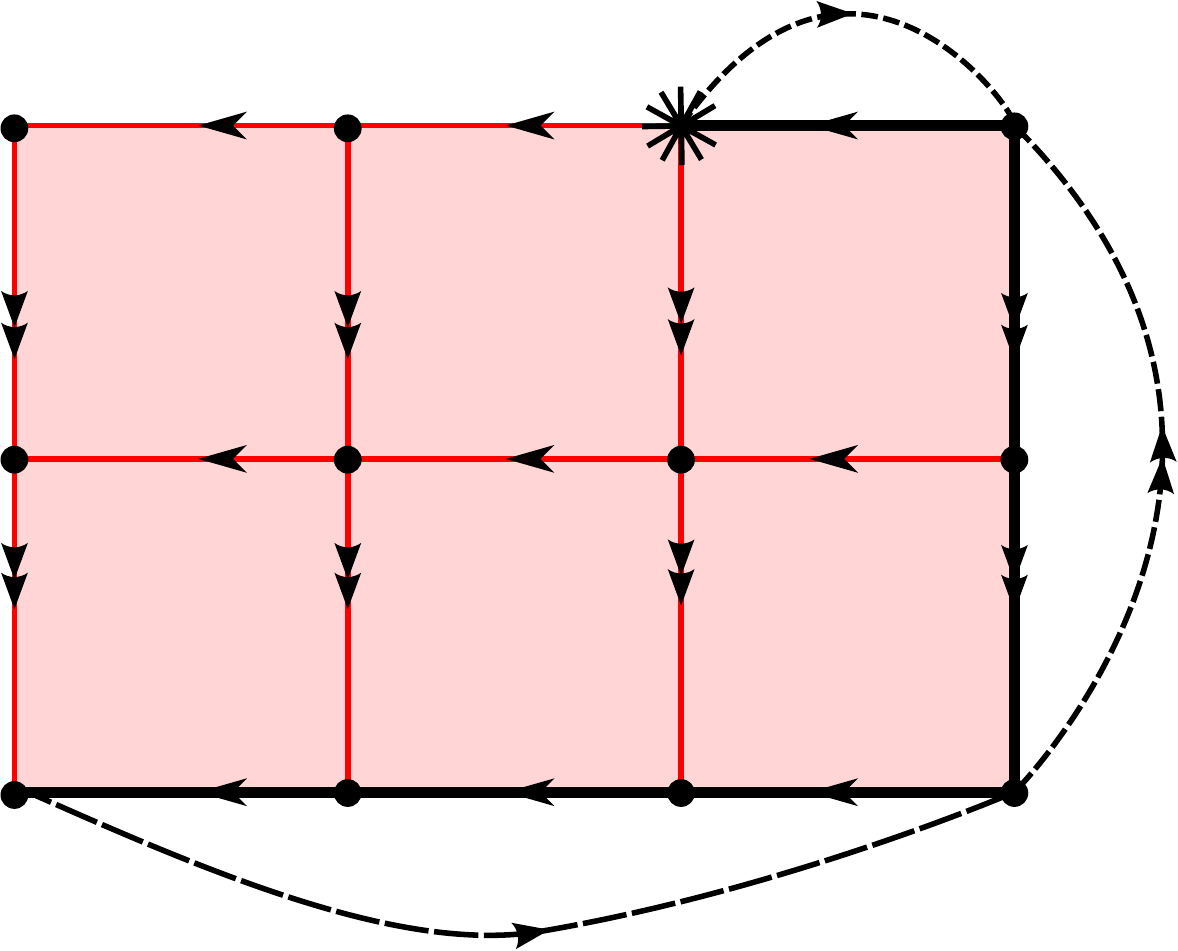}
 \put(49,63){$\mathbf{v_0}$}
 \put(70,60){$\mathbf{D_3}$}
 \put(75,30){$\mathbf{D_2}$}
 \put(41,15.5){$\mathbf{D_1}$}
 \put(41,4){$\mathbf{\gamma_1}$}
 \put(92,61){$\mathbf{\gamma_2}$}
 \put(72,81){$\mathbf{\gamma_3}$}

\end{overpic}
\caption{The map $f$ fails to be locally injective at $v_0$.}\label{fig:figure_p2}
\end{minipage}
\hspace{-1cm}
\begin{minipage}[b]{.4\textwidth}
\centering
\begin{overpic}[width=0.7\textwidth]{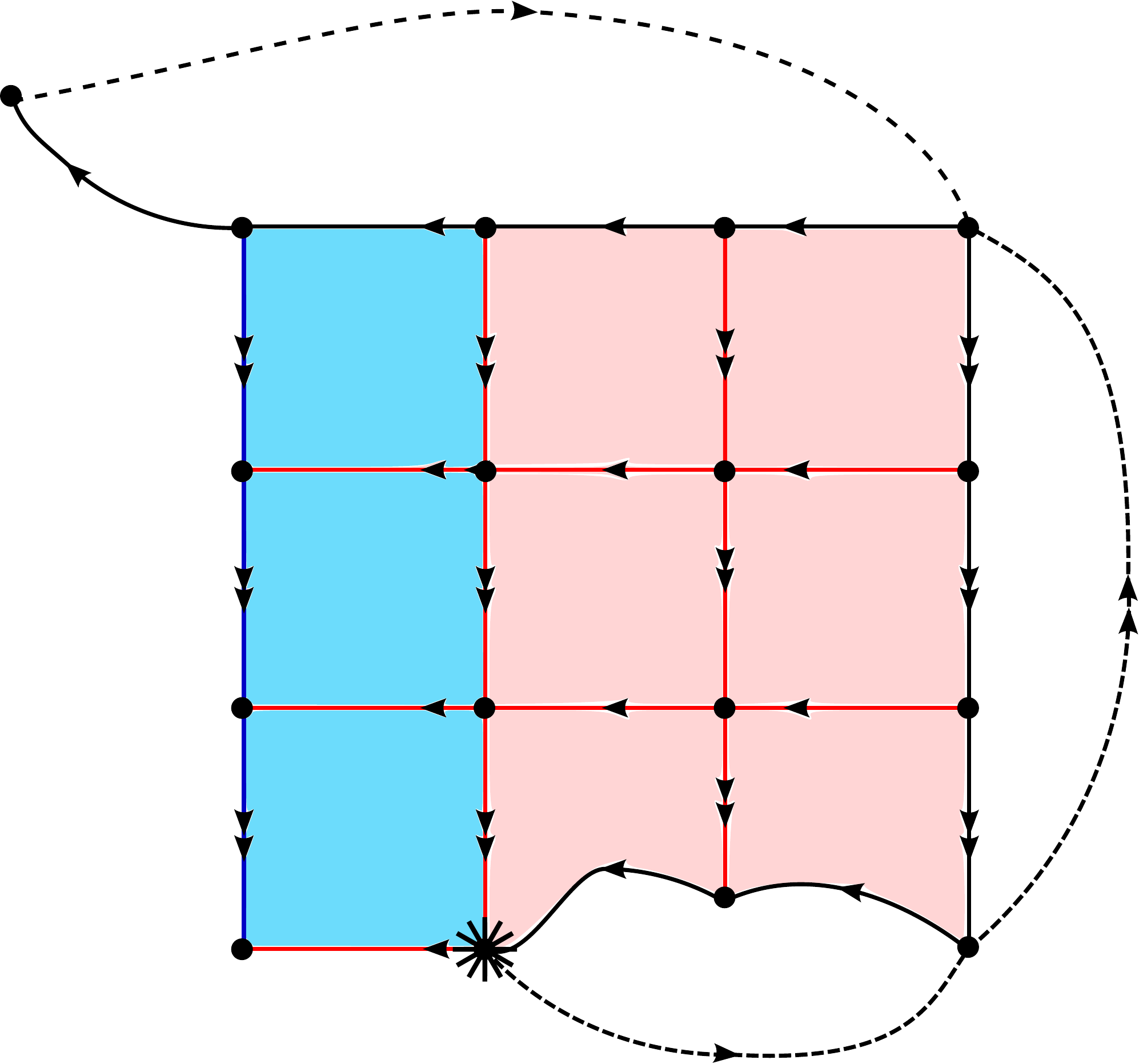}
\put(50,75){$\mathbf{D_3}$}
\put(75,55){$\mathbf{D_2}$}
\put(50,18){$\mathbf{D_1}$}
\put(33,12){$\mathbf{v_1}$}
\put(54,4){$\mathbf{\gamma_1}$}
\put(99,50){$\mathbf{\gamma_2}$}
\put(50,96){$\mathbf{\gamma_3}$}
\end{overpic}
\caption{The map $f$ fails to be locally injective at $v_1$.}\label{fig:figure_p3}
\end{minipage}
\end{figure}

A similar argument shows that the fourth side of $D_1 \times D_2$ must coincide with either $\gamma_4$, in which case we set $D_4= \gamma_4$, or $C_4^c$ so that we set $D_4 = C_4^c$. Again, we can show that $\ell_2 = \ell_4$ so that $D_4$ forms the fourth side of $D_1 \times D_2$.

By construction, each $D_i$ shares endpoints with the $\gamma_i$ for $i = 1, 2, 3, 4$. Thus, $\gamma = \gamma_1 \gamma_2 \gamma_3 \gamma_4$ forms a closed path since $D_1D_2D_3D_4$ forms the closed boundary path of the rectangle $D_1 \times D_2$.

{\bf Case 1b:} Suppose that $\gamma_1$ is not a loop and that $\gamma_2$ is a loop. In this case, $\gamma_3$ necessarily equals $\gamma_1^{-1}$ since $\gamma$ is a lift of a path in $S_\Gamma$ representing a commutator relation in $A_\Gamma$.

If $\gamma_4$ is also a loop, then $\gamma$ is a closed path. Hence suppose that $\gamma_4$ is not a loop. Then either $\gamma_i \subset Y$ and we set $D_i = \gamma_i$, or $D_i = C_i^c$  for $ i = 1, 4$. Thus the map $D_1 \cup D_{4} \hookrightarrow Y$ extends to a map $D_1 \times D_{4} \hookrightarrow Y$, so that we have an embedded rectangle $D_1 \times D_4$ in $Y$. Now, a similar argument as in the previous case shows that $\gamma_2$ must be the third side of $D_1 \times D_4$, contradicting the fact that $D_1 \times D_4$ is embedded. Therefore, $\gamma_4$ must be a loop and $\gamma$ is a closed path in $Y^\circ$.

{\bf Case 2:} Suppose that $\gamma_1$ is a loop in $Y^\circ$. If $\gamma_2$ is also a loop, $\gamma$ is of course a closed path in $Y^\circ$. Therefore, we assume that $\gamma_2$ is not a loop. Note that if $\gamma_3$ is a loop, then $\gamma_4$ necessarily equals $\gamma_2^{-1}$ and $\gamma$ is a closed path.

Assume for contradiction that $\gamma_3$ is not a loop. Then either $\gamma_i \subset Y$ and we set $D_i = \gamma_i$, or $D_i = C_i^c$  for $ i = 2, 3$, and thus, we have an embedded rectangle $D_2 \times D_{3}$ in $Y$. As in the previous case, $\gamma_1$ must form another side of $D_2 \times D_3$, contradicting the fact that $D_2 \times D_3$ is embedded. Thus, $\gamma_3$ must be a loop and $\gamma$ is a closed path in $Y^\circ$.
\end{proof}

The following simple observation plays a crucial role in the proof of Theorem~\ref{thm:introraag}.

\begin{lem}\label{lem:cancom_one_step}
The nonpositively-curved cube complex $\cancom{Y}{S_{\Gamma}}$ is connected when $Y$ is connected, and $\left|\cancom{Y}{S_{\Gamma}}^{(0)}\right|=|Y^{(0)}|$.  Hence $\deg\left(\cancom{Y}{S_{\Gamma}}\rightarrow S_{\Gamma}\right)=|Y^{(0)}|$.
\end{lem}

\begin{proof}
The first assertion is immediate from the construction of $\cancom{Y}{S_{\Gamma}}$.  The second follows from the fact that $\cancom{Y}{S_{\Gamma}}$ contains $Y$ and does not contain any 0-cube not in $Y$.  This, together with the fact that $S_{\Gamma}$ has a single 0-cube, implies the third assertion.
\end{proof}

\subsubsection{Structure of $\widetilde S_{\Gamma}$}\label{subsec:standard_flats}
Let $\Gamma$ be a finite simplicial graph, and let $S_{\Gamma}$ be the Salvetti complex of $A_{\Gamma}$.  Recall that for each cube $c\rightarrow S_{\Gamma}$, the attaching map identifies opposite faces of $c$, so that the image of $c$ is an embedded $\dimension c$-torus.  Such a torus is a \emph{standard torus} of $S_{\Gamma}$, and a standard torus $T\subseteq S_{\Gamma}$ is \emph{maximal} if it is not properly contained in a standard torus.  (We emphasize that 0-cubes and 1-cubes in $S_{\Gamma}$ are also standard tori.)  The inclusion $T_n\rightarrow S_{\Gamma}$ of the standard $n$-torus $T_n$ lifts to an isometric embedding $\widetilde T_n\rightarrow\widetilde S_{\Gamma}$ of universal covers.  In fact, $\widetilde T_n$ has a natural CAT(0) cubical structure obtained by pulling back the cell structure on $T_n$: as a cube complex, $\widetilde T_n$ is the standard tiling of $\Euclidean^n$ by unit $n$-cubes.  Such a subcomplex $\widetilde T_n\subseteq\widetilde S_{\Gamma}$ is a \emph{standard flat} (and a \emph{maximal standard flat} if $T_n$ is a maximal standard torus).  Since the inclusion $T_n\hookrightarrow S_{\Gamma}$ is easily seen to be a local isometry, $\widetilde T_n\subseteq\widetilde S_{\Gamma}$ is a convex subcomplex.

\section{Virtually special groups}\label{sec:separability}

This section presents a proof of Theorem \ref{thm:ssraag}. To this end, let $\Gamma$ be a simplicial graph, let $A_\Gamma$ be the corresponding raAg, and let $S_\Gamma$ be the corresponding Salvetti complex.  The \emph{label} of a 1-cube $e$ of $\widetilde S_\Gamma$ is the 1-cube of $S_\Gamma$ to which $e$ maps.  For each hyperplane $H$ of $\widetilde S_\Gamma$, the 1-cubes dual to $H$ all have the same label, which we call the \emph{label} of $H$.  Let $H$ be labeled by $a$.  Then $\widetilde S_\Gamma$ has a convex subcomplex $P(H) = H\times L_a$, where $L_a$ is a convex combinatorial line all of whose 1-cubes are labelled by $a$.

\newcommand{\fram}[2]{\mathcal F_{#2}(#1)}

\begin{lem}\label{lem:intersect}
Let $K\subset\widetilde S_\Gamma$ be a convex subcomplex and let $H$ be a hyperplane such that $H\cap K\neq\emptyset$. Then $P(H)\cap K=(H\cap K)\times L'_a$, where $L'_a$ is a combinatorial subinterval of $L_a$ of length at least one.
\end{lem}

\begin{proof}
This follows from convexity of $K$ and Lemma~2.5 of~\cite{CapraceSageev}.
\end{proof}

The complex $\fram{H}{K}=P(H)\cap K$ is the \emph{frame} of $H$ in $K$ and is shown in Figure~\ref{fig:frame}.

\begin{figure}[h]
\begin{overpic}[width=0.5\textwidth]{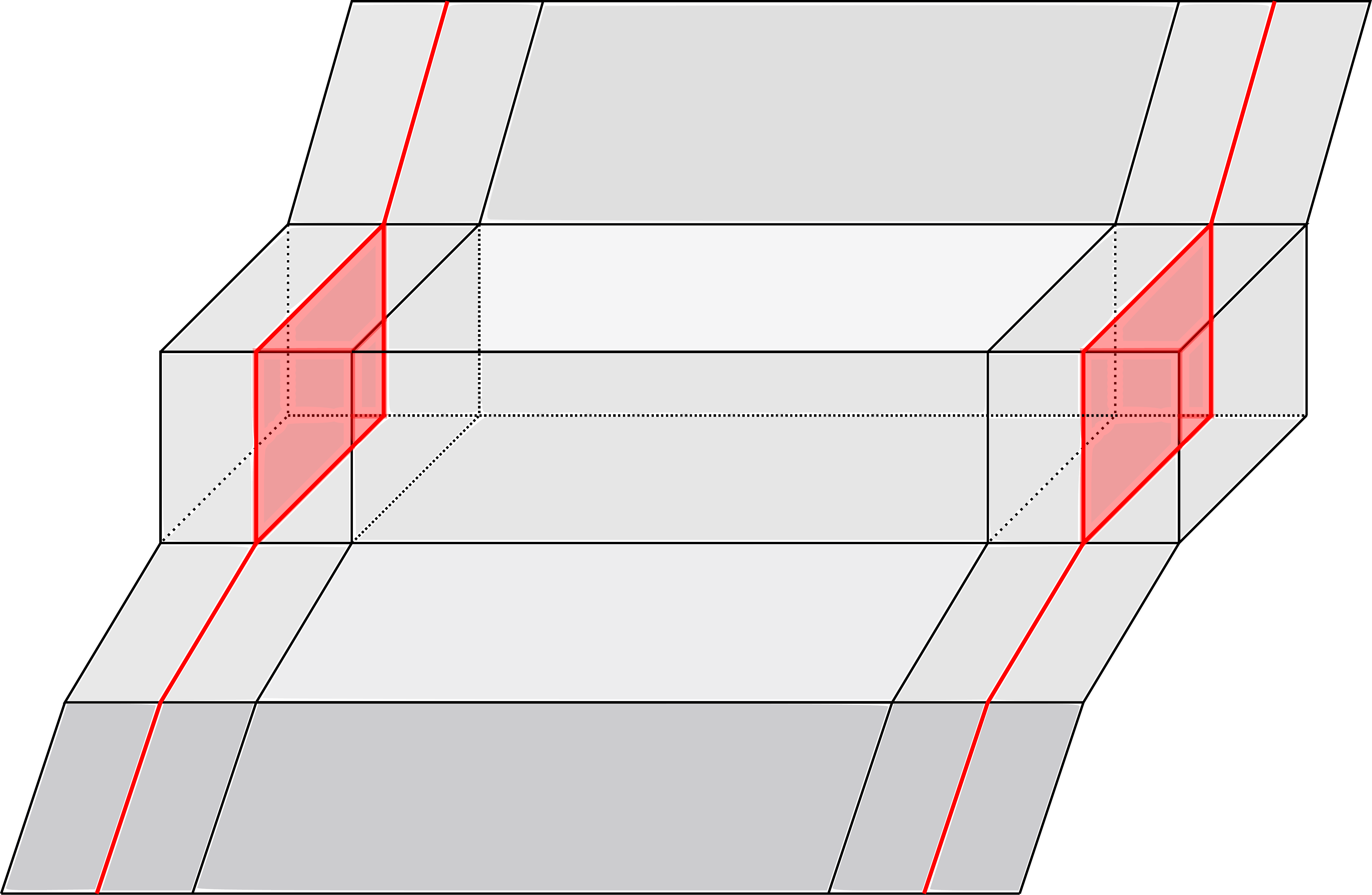}
\end{overpic}
\caption{A frame showing {\color{red}two translates} of the hyperplane $H$.}\label{fig:frame}
\end{figure}

\begin{proof}[Proof of Theorem \ref{thm:ssraag}]
Let $\tilde v$ be a lift of $v$ to $\widetilde S_\Gamma$ and let $g\in A_\Gamma\setminus\{1\}$ and let $\dist(\tilde v, g\tilde v) = n \geq 1$. Let $K$ be the convex hull of $\{\tilde v,g\tilde v\}$.
There exist a set of hyperplanes $H_1,\ldots,H_k$ with the property that each $H_i$ separates $\tilde v,g\tilde v$, such that $H_{i\pm1}\cap K$, for $2\leq i\leq k-1$, lie in two distinct connected components of $K \setminus H_{i}\cap K$.  By passing to a subset if necessary, we can assume that $\fram{H_i}{K}\neq\fram{H_j}{K}$ for $i\neq j$.  For each $i$, we have $\fram{H_i}{K}\cong (H_i\cap K)\times L'_i$, where $L_i'\cong[0,\ell_i]$ with $\ell_i\geq 1$.

By definition of a frame, the fact that $\fram{H_i}{K} \neq \fram{H_{i+1}}{K}$ and that $\stabilizer(H_i)$ is the centralizer of the generator labeling $L_i$,  the labels of $L_i'$ and $L'_{i\pm1}$ are distinct for all $i$.  Moreover, since $H_i\cap H_{i\pm1}=\emptyset$, no 1-cube of $L'_{i\pm1}$ lies in the $(H_i\cap K)$-factor of $\fram{H_i}{K}$ or vice versa.  This fact together with the fact that each hyperplane intersecting $K$ must separate $\tilde v$ from $g\tilde v$, implies that $\fram{H_i}{K}\cap\fram{H_{i+1}}{K}=(H_i\cap K)\times\{\ell_i\}\cap (H_{i+1}\cap K)\times\{0\}$.  Indeed, the intersection involves the $0$ and $\ell_i$ factors only since for any three pairwise non-intersecting hyperplanes of $K$, some pair is separated by the third.
Since $H_i$ separates $H_{i'}$ from $H_{i''}$ when $i'<i<i''$, we have that $\fram{H_i}{K}\cap\fram{H_{i'}}{K}=\emptyset$ if $|i-i'|>1$.

Finally, we can make the above choices so that $\fram{H_i}{K}\cap\fram{H_{i+1}}{K}\neq\emptyset$.  Indeed, were the intersection empty, then by convexity of frames, there would be a hyperplane $H$ separating $\fram{H_i}{K}$ from $\fram{H_{i+1}}{K}$ and hence separating $H_i$ from $H_{i+1}$; such an $H$ could be included in our original sequence and its frame in $K$ in our original sequence of frames.  Moreover, by a similar argument, $\tilde v$ is in the $\fram{H_1}{K}$ and $g \tilde v$ is in $\fram{H_k}{K}$. Hence,  without loss of generality, there is an embedded piecewise-geodesic combinatorial path $\gamma=Q_1L_1'\cdots Q_kL_k'Q_{k+1}$ in $K$ joining $\tilde v$ to $g\tilde v$, where $Q_i\subset (H_i\cap K)\times\{0\}$ for $i\leq k$, $Q_{k+1}\subset (H_{k}\cap K)\times\{\ell_i\}$ and $L_i'$ is chosen within its parallelism class so that the above concatenation exists.  Note that $\sum_i\ell_i\leq n$.

Let $P=\cup_{i=1}^k\fram{H_i}{K}$, so $P$ is connected and contains $\gamma$.  For each $i$, let $\rho_i:\fram{H_i}{K}\rightarrow \overline{H_i\cap K}\times L'_i$ be the cubical quotient induced by identifying the endpoints of each 1-cube of $H_i$ and folding as necessary.  More precisely, for each $i$, we identify the endpoints of each 1-cube of $H_i\cap K$.  This induces a cubical quotient $H_i\rightarrow\overline H_i$.  We then \emph{fold}, i.e. identify cubes $c_1,c_2$ for which $c_1\cup c_2\rightarrow S_\Gamma$ is not locally injective.  (This straightforwadly generalizes Stallings folding for maps of graphs.)  The resulting (folded) quotient is $\overline {H_i\cap K}$, and $\rho_i$ is the induced map acting as the identity on $L'_i$.

Since $\rho_i$ and $\rho_{i+1}$ agree on $\fram{H_i}{K}\cap\fram{H_{i+1}}{K}$, these maps can be pasted together to form a quotient $\rho:P\rightarrow\overline P$ with $\overline P$ a nonpositively-curved cube complex.  Note that the restriction of $\widetilde S_\Gamma\rightarrow S_\Gamma$ descends to a locally injective cubical map $\overline P\rightarrow S_\Gamma$.

We claim that $\rho\circ\gamma$ is a path in $\overline P$ that contains every 0-cube and has distinct endpoints.  This follows from the fact that $\rho(L'_i)\cap\rho(L'_j)$ is a single 0-cube if $|i-j|=1$ and is otherwise empty if $i\neq j$, and $\rho$ is injective on each $L'_i$.  Since $\gamma$ passes through each 1-cube of $\cup_iL'_i$ exactly once, and the image of each $Q_i$ maps to a wedge of circles in $\overline P$, it follows that $\gamma$ has the desired properties.  Hence, $|\overline P^{(0)}|\leq n+1$.

We would like to finish by applying Lemma~\ref{lem:cancom_one_step} to $\overline P$. However, the constructed cube-complex, $\overline P$, is not necessarily locally convex in $S_\Gamma$.
To fix this, let $s= c_i\times c_{i+1}$ be a 2-cube of $K$ such that $c_i$ is a 1-cube in $\fram{H_i}{K}$ and $c_{i+1}$ is a 1-cube in $\fram{H_{i+1}}{K}$, as in Figure~\ref{fig:missing_square}.

\begin{figure}[h]
\begin{overpic}[width=0.25\textwidth]{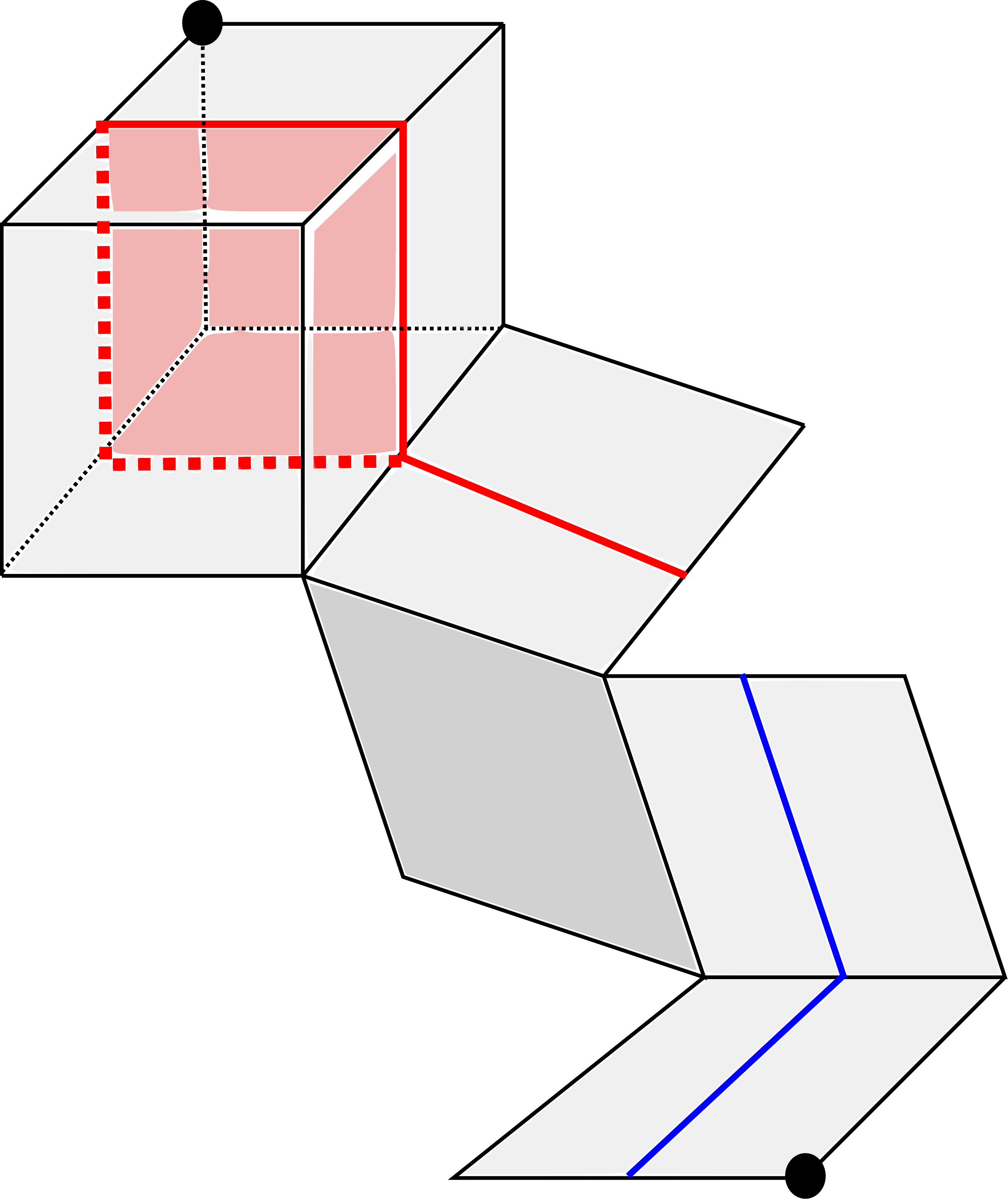}

\end{overpic}
\caption{Two frames and a missing 2-cube.}\label{fig:missing_square}
\end{figure}

Let $\bar s=\rho(c_i)\times\rho(c_{i+1})$.  Without loss of generality, $c_{i+1}\subset H_{i+1}\times\{0\}$.  Indeed, the generators labeling $L'_i$ and $L'_{i+1}$ do not commute, so at most one of $c_i,c_{i+1}$ is in $L_i',L'_{i+1}$.  Hence $\bar s$ is either a cylinder or a torus.  In the latter case, glue $\bar s$ to $\overline P$ along $\rho(c_i)\cup\rho(c_{i+1})$, noting that we do not add 0-cubes in so doing and moreover, we do not add 1-cubes.  Hence no missing corners are introduced.

In the former case, the label of $c_{i+1}$ corresponds to a generator of $A_\Gamma$ that commutes with the generator labeling $L'_i$, and hence $c_{i+1}\subset (H_{i}\cap K)\times\{\ell_i\}\cap (H_{i+1}\cap K)\times\{0\}$. Therefore, $s\subset\fram{H_i}{K}$.  We conclude that the quotient $\rho$ extends to a quotient $K\rightarrow\overline K$ such that $\overline P\subseteq\overline K$ and the restriction of $\widetilde S_\Gamma\rightarrow S_\Gamma$ to $K$ descends to a local isometry $\overline K\rightarrow S_\Gamma$.  Moreover, since $\overline K$ is formed from $\overline P$ by attaching 2-tori as above, and adding higher-dimensional tori when their 2-skeleta appear, we see that $|\overline K^{(0)}|=|\overline P^{(0)}|\leq n+1$.  Hence, $\cancom{\overline K}{S_\Gamma}$ is a cover of $S_\Gamma$ of degree at most $n+1$, by Lemma~\ref{lem:cancom_one_step}, such that $\gamma:[0,n]\rightarrow\widetilde S_\Gamma\rightarrow S_\Gamma$ lifts to a non-closed path in $\cancom{\overline K}{S_\Gamma}$, and the proof is complete.
\end{proof}

\section{Special linear groups} \label{sec:arithmetic}

\subsection{The upper bound}
Fix a generating set for $\SL_k(\Z)$.
Let $g$ be a nontrivial element in the word-metric ball of radius $n$ in $\SL_k(\Z)$.
Since $Z(\SL_k(\Z))$ is finite, we may assume that $g \notin Z(\SL_k(\Z))$.
Thus there exists an off-diagonal entry of $g$ that is not zero or two diagonal entries with non-zero difference. Select $\alpha$ so that it is one of these values and non-zero.
By \cite[Proposition 4.1]{BK12},  there exists some prime $p$ with
$$ p \leq C n$$
for some fixed constant $C$, where the image of $g$ in $\SL_k(\Z/p \Z)$ is not central (that is, $\alpha$ does not vanish in $\Z/p \Z$).
The group $\SL_k(\Z/p\Z)$ has subgroup
$$
\Delta:=\left\{ \begin{pmatrix}
* & \cdots&  * & * \\
\vdots & & \vdots & \vdots \\
* & \cdots & * & *\\
0 & \cdots & 0 & *
\end{pmatrix} \in \SL_k(\Z/p \Z)
: *\text{ entries are arbitrary}\right\}.
$$
Using a dimension counting argument, it is straightforward to see that the index of $\Delta$ in $\SL_k(\Z/p \Z)$ is bounded above by $C' p^{k-1}$ where $C'$ depends only on $k$.
Since $\SL_k(\Z/ p \Z)$ maps onto a simple group with kernel $Z(\SL_k(\Z/ p \Z))$, it follows that the intersection of all conjugates of $\Delta$ is contained in $Z(\SL_k(\Z/ p \Z))$ (note that $\Delta$ contains $Z(\SL_k(\Z/ p \Z))$).
Thus, there exists some conjugate of $\Delta$ that misses the image of $g$, which is not central, in $\SL_k(\Z/ p \Z)$.
Thus, we get
$$
\F_{\SL_k(\Z)} (n) \preceq n^{k-1}.
$$

\subsection{The lower bound}
Here, we show that the residual finiteness growth of $\SL_k(\Z)$, $k > 2$, is bounded below by $n^{k-1}$.
Before we get into the proof, we need a lemma involving unipotent subgroups of $\SL_k(\Z)$.
Let $E_{i,j}(\alpha)$ be the elementary matrix with $\alpha$ in the $i$th row and $j$th column.

\begin{lem} \label{lem:heisenberg}
Let $H$ be the subgroup of $\SL_k(\Z)$ that is the $2k-1$ dimensional generalized Heisenberg group.
Set $g_n = E_{1,k}(\lcm(1, \ldots, n))$.
Then $D_H(g_n) \geq n^{k-1}.$
\end{lem}

\begin{proof}
Let $\Delta$ be a finite-index subgroup of $H$ that does not contain $g_n$.
Set $d = 2k-3$.
By using $E_{i,j}(1)$ as a Mal'cev basis we may associate to $\Delta$ a matrix $\{ m_{i,j} \}$ (see \cite[Lemma 2.3]{MR943928})
with $[\Gamma : \Delta] = \prod_{i=1}^d m_{i,i}$ where $(E_{1,k}(1))^{m_{1,1}} \in \Delta$ and, in fact, we have $k-2$ conditions:
\begin{eqnarray*}
m_{1,1} &\text{ divides }& m_{2,2} m_{k, k}, \\
m_{1,1} &\text{ divides }& m_{3,3} m_{k+1, k+1}, \\
& \vdots &  \\
m_{1,1} &\text{ divides }& m_{k-1,k-1} m_{2k-3, 2k-3}.
\end{eqnarray*}
Thus, $\prod_{i=1}^d m_{i,i} \geq m_{1,1}^{k-1}$.
As $g_n \notin \Delta$, we have that $m_{1,1}$ does not divide $\lcm(1, \ldots, n)$, i.e. $m_{1,1} > n$, so
$$
D_H(g_n) \geq n^{k-1},
$$
as desired.
\end{proof}

We can now prove the lower bound.
We begin by following the first part of the proof of \cite[Theorem 2.6]{B09}.
By \cite[Theorem A]{LMR00}, there exists a finite generating set, $S$, for $\SL_k(\Z)$ (see also Riley \cite{R05}) and a
$C > 0$ satisfying
$$
\| - \|_S \leq C \log( \| - \|_1 ),
$$
where $\| - \|_1$ is the 1-operator norm for matrices.
Thus, as
$\log( \| E_{1,k}(\lcm(1, \ldots, n))  \|_1) = \log(\lcm(1,\ldots, n)) +1 \approx n$
by the prime number theorem, the elementary matrix may be written in terms of at most $Cn$ elements from $S$.
The matrix $g_n := E_{1,k}(\lcm(1, \ldots, n))$ in $\SL_k(\Z)$ is our candidate.

$g_n$ is contained in $H \leq \SL_k(\Z)$ as in Lemma \ref{lem:heisenberg}.
It follows then that
$$
D_{\SL_k(\Z)} (g_n) \geq D_{H} (g_n) \geq n^{k-1}.
$$
Thus, $\F_{\SL_k(\Z)} (n ) \succeq n^{k-1},$ as desired.

\bibliography{refs}

\def\cprime{$'$} \def\cprime{$'$} \def\cprime{$'$} \def\cprime{$'$}
\begin{thebibliography}{AFW13}

\bibitem[AFW13]{Wilton}
M.~Aschenbrenner, S.~Friedl, and H.~Wilton.
\newblock {3-}manifold groups.
\newblock {\em Ar{X}iv 1205.0202}, pages 1--149, 2013.

\bibitem[Ago12]{Agol:virtual_haken}
Ian Agol.
\newblock The virtual {H}aken conjecture.
\newblock {\em arXiv:1204.2810}, 2012.
\newblock Primary article by Ian Agol, with an appendix by Ian Agol, Daniel
  Groves, and Jason Manning.

\bibitem[Big01]{MR1815219}
Stephen~J. Bigelow.
\newblock Braid groups are linear.
\newblock {\em J. Amer. Math. Soc.}, 14(2):471--486 (electronic), 2001.

\bibitem[BR10]{B09}
Khalid Bou-Rabee.
\newblock Quantifying residual finiteness.
\newblock {\em J. Algebra}, 323(3):729--737, 2010.

\bibitem[BR11a]{MR2851069}
Khalid Bou-Rabee.
\newblock Approximating a group by its solvable quotients.
\newblock {\em New York J. Math.}, 17:699--712, 2011.

\bibitem[BR11b]{B11}
Khalid Bou-Rabee.
\newblock Approximating a group by its solvable quotients.
\newblock {\em New York J. Math.}, 17:699--712, 2011.

\bibitem[Bri91]{bridson}
M.R. Bridson.
\newblock Geodesics and curvature in metric simplicial complexes.
\newblock In E.~Ghys, A.~Haefliger, and A.~Verjovsky, editors, {\em Group
  theory from a geometrical viewpoint}, Proc. {I}{C}{T}{P}, {T}rieste, {I}taly,
  pages 373--463. World Scientific, Singapore, 1991.

\bibitem[BRK12]{BK12}
Khalid Bou-Rabee and Tasho Kaletha.
\newblock Quantifying residual finiteness of arithmetic groups.
\newblock {\em Compos. Math.}, 148(3):907--920, 2012.

\bibitem[BRM]{BM13}
K.~Bou-Rabee and D.~B. McReynolds.
\newblock Extremal behavior of divisibility functions.
\newblock {\em (submitted)}.

\bibitem[BRM10]{BM10}
K.~Bou-Rabee and D.~B. McReynolds.
\newblock Bertrand's postulate and subgroup growth.
\newblock {\em J. Algebra}, 324(4):793--819, 2010.

\bibitem[BRM11]{BM11}
K.~Bou-Rabee and D.~B. McReynolds.
\newblock Asymptotic growth and least common multiples in groups.
\newblock {\em Bull. Lond. Math. Soc.}, 43(6):1059--1068, 2011.

\bibitem[Bus09]{MR2583614}
N.~V. Buskin.
\newblock Efficient separability in free groups.
\newblock {\em Sibirsk. Mat. Zh.}, 50(4):765--771, 2009.

\bibitem[CD95]{charney_davis}
Ruth Charney and Michael~W. Davis.
\newblock Finite {$K(\pi, 1)$}s for {A}rtin groups.
\newblock In {\em Prospects in topology (Princeton, NJ, 1994)}, volume 138 of
  {\em Ann. of Math. Stud.}, pages 110--124. Princeton Univ. Press, Princeton,
  NJ, 1995.

\bibitem[Cha07]{charney_raag_survey}
Ruth Charney.
\newblock An introduction to right-angled {A}rtin groups.
\newblock {\em Geometriae Dedicata}, 125:141--158, 2007.

\bibitem[Che00]{Chepoi:cube_median}
Victor Chepoi.
\newblock Graphs of some {${\rm CAT}(0)$} complexes.
\newblock {\em Adv. in Appl. Math.}, 24(2):125--179, 2000.

\bibitem[CS11]{CapraceSageev}
Pierre-Emmanuel Caprace and Michah Sageev.
\newblock Rank rigidity for {CAT}(0) cube complexes.
\newblock {\em Geom. Funct. Anal.}, 21:851--891, 2011.

\bibitem[CW02]{MR1942303}
Arjeh~M. Cohen and David~B. Wales.
\newblock Linearity of {A}rtin groups of finite type.
\newblock {\em Israel J. Math.}, 131:101--123, 2002.

\bibitem[Gro87]{gromov}
M.~Gromov.
\newblock Hyperbolic groups.
\newblock In {\em Essays in group theory}, volume~8 of {\em Math. Sci. Res.
  Inst. Publ.}, pages 75--263. Springer, New York, 1987.

\bibitem[GSS88]{MR943928}
F.~J. Grunewald, D.~Segal, and G.~C. Smith.
\newblock Subgroups of finite index in nilpotent groups.
\newblock {\em Invent. Math.}, 93(1):185--223, 1988.

\bibitem[Hag12]{HagenPhD}
M.F. Hagen.
\newblock {\em Geometry and combinatorics of cube complexes}.
\newblock PhD thesis, McGill University, 2012.

\bibitem[HW08]{haglundwise}
Fr{\'e}d{\'e}ric Haglund and Daniel~T. Wise.
\newblock Special cube complexes.
\newblock {\em Geom. Funct. Anal.}, 17(5):1 551--1620, 2008.

\bibitem[HW10]{MR2646113}
Fr{\'e}d{\'e}ric Haglund and Daniel~T. Wise.
\newblock Coxeter groups are virtually special.
\newblock {\em Adv. Math.}, 224(5):1890--1903, 2010.

\bibitem[HW12]{haglund_wise_amalgams}
Fr\'{e}d\'{e}ric Haglund and Daniel~T. Wise.
\newblock A combination theorem for special cube complexes.
\newblock {\em Ann. of Math. (2)}, 176(3):1427--1482, 2012.

\bibitem[HW13]{hagen_wise_irred}
M.F. Hagen and D.T. Wise.
\newblock Cubulating hyperbolic free-by-cyclic groups: the irreducible case.
\newblock {\em ar{X}iv 1311.2084}, pages 1--39, 2013.

\bibitem[IBM]{BBKM13b}
M.~Kassabov I.~Biringer, K.~Bou-Rabee and F.~Matucci.
\newblock Intersection growth in groups.
\newblock in preparation.

\bibitem[KM11]{MR2784792}
Martin Kassabov and Francesco Matucci.
\newblock Bounding the residual finiteness of free groups.
\newblock {\em Proc. Amer. Math. Soc.}, 139(7):2281--2286, 2011.

\bibitem[Kra02]{MR1888796}
Daan Krammer.
\newblock Braid groups are linear.
\newblock {\em Ann. of Math. (2)}, 155(1):131--156, 2002.

\bibitem[Lea10]{leary}
Ian Leary.
\newblock A metric {K}an-{T}hurston theorem.
\newblock {\em Preprint}, 2010.

\bibitem[Liu]{liu_graph}
Yi~Liu.
\newblock Virtual cubulation of nonpositively curved graph manifolds.

\bibitem[LMR00]{LMR00}
Alexander Lubotzky, Shahar Mozes, and M.~S. Raghunathan.
\newblock The word and {R}iemannian metrics on lattices of semisimple groups.
\newblock {\em Inst. Hautes \'Etudes Sci. Publ. Math.}, (91):5--53 (2001),
  2000.

\bibitem[NR03]{MR1983376}
G.~A. Niblo and L.~D. Reeves.
\newblock Coxeter groups act on {${\rm CAT}(0)$} cube complexes.
\newblock {\em J. Group Theory}, 6(3):399--413, 2003.

\bibitem[OW11]{OllivierWise:Density}
Yann Ollivier and Daniel~T. Wise.
\newblock Cubulating random groups at density~$<\frac16$.
\newblock {\em Trans. Amer. Math. Soc.}, 363:4701--4733, 2011.

\bibitem[Pat12]{PP12}
Priyam Patel.
\newblock On a theorem of {P}eter {S}cott.
\newblock {\em Proc. Amer. Math. Soc.}, 2012.
\newblock To appear.

\bibitem[PW]{PrzytyckiWise:mixed}
Piotr Przytycki and Daniel~T. Wise.
\newblock Mixed 3-manifolds are virtually special.
\newblock pages 1--24.
\newblock Available at arXiv:1205.6742.

\bibitem[PW11]{PrzytyckiWise:cube_GraphManifold}
Piotr Przytycki and Daniel~T. Wise.
\newblock Graph manifolds with boundary are virtually special.
\newblock {\em Ar{X}iv eprint 1110.3513}, page~24, 2011.

\bibitem[Ril05]{R05}
T.~R. Riley.
\newblock Navigating in the {C}ayley graphs of {${\rm SL}_N(\Bbb Z)$} and
  {${\rm SL}_N(\Bbb F_p)$}.
\newblock {\em Geom. Dedicata}, 113:215--229, 2005.

\bibitem[Riv12]{R12}
Igor Rivin.
\newblock Geodesics with one self-intersection, and other stories.
\newblock {\em Adv. Math.}, 231(5):2391--2412, 2012.

\bibitem[Sag95]{Sageev:cubes_95}
Michah Sageev.
\newblock Ends of group pairs and non-positively curved cube complexes.
\newblock {\em Proc. London Math. Soc. (3)}, 71(3):585--617, 1995.

\bibitem[Sta83]{stallings}
John~R. Stallings.
\newblock Topology of finite graphs.
\newblock {\em Inventiones mathematicae}, 71(3):551--565, 1983.

\bibitem[Wis]{Wise:QCH}
Daniel~T. Wise.
\newblock The structure of groups with a quasiconvex hierarchy.
\newblock 205 pp. Preprint 2011.

\bibitem[Wis12]{cbmsnotes}
Daniel~T. Wise.
\newblock From riches to {R}{A}{A}{G}s: 3-manifolds, right-angled {A}rtin
  groups, and cubical geometry.
\newblock In {\em Lecture notes}, {N}{S}{F}-{C}{B}{M}{S} Conference,
  {C}{U}{N}{Y} {G}raduate {C}enter, {N}ew {Y}ork, 2012.

\end{thebibliography}
\bibliographystyle{alpha}

\end{document}